\documentclass[11pt,leqno,oneside]{amsart}
\usepackage{layout}
\usepackage{mathrsfs,dsfont}
\usepackage{amsmath,amstext,amsthm,amssymb,bbm,color}
\usepackage{comment}
\usepackage{charter}
\usepackage{typearea}
\usepackage{pdfsync}
\usepackage[width=6.5in,height=8.7in]{geometry}

\def\bt{\begin{thm}}
\def\et{\end{thm}}
\def\bl{\begin{lem}}
\def\el{\end{lem}}
\def\bd{\begin{defn}}
\def\ed{\end{defn}}
\def\bc{\begin{cor}}
\def\ec{\end{cor}}
\def\bp{\begin{proof}}
\def\ep{\end{proof}}
\def\br{\begin{rem}}
\def\er{\end{rem}}

\newtheorem{thm}{Theorem}[section]
\newtheorem{prop}[thm]{Proposition}
\newtheorem{lem}[thm]{Lemma}
\newtheorem{defn}[thm]{Definition}

\newtheorem{rem}[thm]{Remark}
\newtheorem{cor}[thm]{Corollary}

\numberwithin{equation}{section}

\newcommand{\bthm}{\begin{thm}}
\newcommand{\ethm}{\end{thm}}
\newcommand{\bstp}{\begin{stp}}
\newcommand{\estp}{\end{stp}}
\newcommand{\blemma}{\begin{lemma}}
\newcommand{\elemma}{\end{lemma}}
\newcommand{\bprop}{\begin{prop}}
\newcommand{\eprop}{\end{prop}}
\newcommand{\bpf}{\begin{pf}}
\newcommand{\epf}{\end{pf}}
\newcommand{\bdefn}{\begin{defn}}
\newcommand{\edefn}{\end{defn}}
\newcommand{\brk}{\begin{rmrk}}
\newcommand{\erk}{\end{rmrk}}
\newcommand{\bcrl}{\begin{crl}}
\newcommand{\ecrl}{\end{crl}}

\title[]{Equidistribution of zeros of Random polynomials and Random Polynomial mappings on $\mathbb{C}^{m}$}

\address{}

\email{ozangunyuz@alumni.sabanciuniv.edu}

\date{}


\begin{document}
\author{Ozan Günyüz}

\address{Faculty of Engineering and Natural Sciences, Sabanc{\i} University, \.{I}stanbul, Turkey}

\begin{abstract}
We study equidistribution problem of zeros in relation to a sequence of $Z$-asymptotically Chebyshev polynomials on $\mathbb{C}^{m}$. We use certain results obtained in a very recent work of Bayraktar, Bloom and Levenberg and have an equidistribution result in a more general probabilistic setting than what the paper of Bayraktar, Bloom and Levenberg considers even though the basis polynomials they use are more general than $Z$-asymptotically Chebyshev polynomials. Our equidistribution result is based on the expected distribution and the variance estimate of random zero currents corresponding to the zero sets (zero divisors) of polynomials. This equidistribution result of general nature shows that equidistribution result turns out to be true without the random coefficients being i.i.d. (independent and identically distributed), which also means that there is no need to use any specific probability distribution function for these random coefficients. In the last section, unlike from the $1$-codimensional case, we study the basis of polynomials orthogonal with respect to the $L^{2}$-inner product defined by the weighted asymptotically Bernstein-Markov measures on a given locally regular compact set, and with a probability distribution studied well by Bayraktar including the (standard) Gaussian and the Fubini-Study probability distributions as special cases, we have an equidistribution result for codimensions bigger than $1$.
  \end{abstract}
\maketitle

%%%%%%%%%%%%%%%%%%%%%%%%%%%%%%%%%%%%%%%%%%%%%%%%%%%%%%%%%%%%%%%%%%%%%%%%%%%%%%%

\section{Introduction and Background}\label{S2}

The statistical issues of zero sets of random functions of several variables, such as random polynomials of multivariable real and complex variables, have piqued the interest of many researchers. It is impossible to give all the references here since there is an extensive literature in this subject. For this reason we will be a little short for explaining what has been done so far. Many results regarding both Gaussian and non-Gaussian cases and historical advancements of this polynomial theory may be found, for example, in \cite{ Bay17b, BL15, BloomS, BL05, BloomD, ROJ, SHSM, HN08} (and references therein). For example, in \cite{BL15}, the authors work with the complex
random variables that have bounded distribution functions on the whole complex plane $\mathbb{C}$ and outside
of a very large disk with radius $\rho$, its integral with respect to the two dimensional Lebesgue
measure has an upper bound depending on $\rho$, the latter condition is called the tail-end estimate. Long before these advances, as is commonly known, the works of Polya and Bloch, Littlewood-Offord, Kac, Hammersley, and Erdös-Turan were the first efforts on the distribution of roots of random algebraic equations in single real variable, and the interested reader can go to the articles \cite{BlP, Kac43, LO43, HAM56, ET50}.

As another interesting direction, there is an expanding physics literature dealing with the equidistribution and probabilistic problems concerned with the zeros of complex random polynomials. See, for example, \cite{FH, Hann, NV98} for foundational research in this area.

As the most general setting so far, the equidistribution, expected distribution and variance of zero currents of integration of random holomorphic sections with different probabilistic settings (including Gaussian and non-Gaussian types) are studied in \cite{BCM, Bay16, Shif, SZ99}. The initial and pioneering work (\cite{SZ99}) in this setting belongs to Shiffman and Zelditch. In this paper, we will prove an equidistribution result and the techniques are based on the papers of \cite{SZ99} and \cite{Shif}. The main tools are expected distribution and variance estimation of currents of integration related to the zero sets of polynomials. In \cite{Bay16}, \cite{Shif}, \cite{SZ08}, \cite{SZ10}, \cite{Gun}, \cite{BG}, the reader can find similar estimations in different probabilistic setups in the general setting of holomorphic line bundles over Kahler manifolds.

The \emph{pluricomplex Green function} of a non-pluripolar compact set $K\subset \mathbb{C}^{m}$ is defined as follows
\begin{equation*}
V_{K}(z):=\sup \{u(z):u|_{K}\leq 0,\ u\in
\mathcal{L}(\mathbb{C}^{m})\},
\end{equation*}%
where \thinspace $\mathcal{L}(\mathbb{C}^{m})$ represents the Lelong class consisting of
all functions\thinspace\ $u$ plurisubharmonic on $\mathbb{C}^{m}$ such that $u(\zeta
)-\ln |\zeta |$ is bounded from above near infinity. The upper semicontinuous regularization of $V_{K}(z)$ is the following $$V^{*}_{K}(z):=\limsup_{\zeta \rightarrow z}{V_{K}(\zeta)}.$$ As is well-known, $V^{*}_{K}(z)\in \mathcal{L}(\mathbb{C}^{m})$
(precisely if $K$ is non-pluripolar, see Corollary 5.2.2 of \cite{Kl}). For more detail about the pluricomplex Green function, \cite{Kl} may be useful.

A compact set $K$ in $\mathbb{C}^{m}$ is \textit{regular} if $V_{K}\equiv 0$ on $K$ (and therefore $V_{K}$ is continuous on $\mathbb{C}^{m}$). The compact sets we consider in this paper will be assumed to be regular.

We use the notation $ ||f|| _{D}:=\sup\left\{ \left\vert f\left( z\right) \right\vert :z\in D\right\} $ for a
function $f:D\rightarrow \mathbb{C}$. Let \thinspace $\mathbb{N}^{m}$\thinspace\ be the collection of
all $m$-dimensional vectors with non-negative integer coordinates. For
\thinspace $k=\left( k_{1},\ldots ,k_{\nu },\ldots ,k_{m}\right) \in \mathbb{
N}^{m}$ and $z=(z_{1},\ldots ,z_{m})\in \mathbb{C}^{m}$, let $%
z^{k(j)}=z_{1}^{k_{1}(j)}\ldots z_{n}^{k_{m}(j)}, \,j\in \mathbb{N}$ and ${\lvert k(j)\rvert }%
:=k_{1}(j)+\ldots +k_{m}(j)$ be the degree of the monomial $e_{j}(z)=z^{k(j)}$. We consider $\
$ the enumeration $\left\{ k\left( j\right) \right\} _{j\in \mathbb{N}}$ of
the set $\mathbb{N}^{m}$ such that ${\lvert k(j)\rvert }\leq {%
\lvert k(j+1)\rvert }$ and on each set $\left\{ \left\vert k\left( j\right)
\right\vert =n\right\} $ the enumeration coincides with the lexicographic
order. We will write $s(j):={\lvert
k(j)\rvert }$. The number of multiindices of degree at most $n$ is
\thinspace $d_{n}:=C_{m+n}^{n}=\dim{(\mathcal{P}_{n})}$, where $\mathcal{P}_{n}$ is the vector space of holomorphic polynomials on $\mathbb{C}^{m}$ of degree at most $n$. \thinspace

The standard $\left( m-1\right) $-simplex will be taken into consideration
\begin{equation}
\Delta :=\left\{ \theta =\left( \theta _{\nu }\right) \in \mathbb{R}%
^{m}:\theta _{\nu }\geq 0,\ \nu =1,\ldots ,m;\ \sum_{\nu =1}^{m}\theta _{\nu
}=1\right\},  \label{sgm}
\end{equation}%
and its interior (with respect to the relative topology on the hyperplane containing $%
\Delta $)
\begin{equation*}
\Delta ^{\circ }:=\left\{ \theta =\left( \theta _{\nu }\right) \in \Delta
:\theta _{\nu }>0,\ \nu =1,\ldots ,m\right\} .
\end{equation*}%
For $\theta \in \Delta$ we denote by $\mathcal{C}_{\theta }$ the set of all
infinite sequences $N\subset \mathbb{N}$ such that $\frac{k\left( j\right) }{%
s\left( j\right) }\overset{N}{\rightarrow }\theta $.

Leja raised the problem as to whether there is usual limit for transfinite diameter in several complex variables (\cite{L}). Zakharyuta in his seminal work \cite{Za1} solved this problem affirmatively for an arbitrary compact set $K\subset \mathbb{C}^n$ by introducing the following what is called \emph{directional Chebyshev constants} \begin{eqnarray}
\tau \left( K,\theta \right) &:&=\limsup\limits_{\substack{ j\rightarrow
\infty  \\ \frac{k\left( j\right) }{\left\vert k\left( j\right) \right\vert }%
\rightarrow \theta }}\tau _{j}:=\sup_{L\in \mathcal{L}_{\theta
}}\limsup\limits_{j\in Y}\tau _{j},\ \theta \in \Delta ,\ \   \label{tkt} \\
\tau _{j} &=&\tau _{j}\left( K\right) :=\left( M_{j}\right) ^{1/s\left(
j\right) },\ j\in \mathbb{N},
\end{eqnarray}%
where
\begin{equation}
M_{j}:=\inf \left\{ \left\vert p\right\vert
_{K}:p=e_{j}+\sum_{l=1}^{j-1}c_{l}\ e_{l}\ \right\} ,\ j\in \mathbb{N}
\label{mii}
\end{equation}The constants $M_{j}$ are known as the \emph{least uniform deviation of monic polynomials} from the identical zero on compact set $K$. A polynomial which attains its infimum in (\ref{mii}) is called a \emph{Chebyshev polynomial}. In the context of the theory of best approximation in Banach spaces (\cite{Ah}, section 8), this kind of polynomials always exists, but the uniqueness is not ensured.

Let $P(k(j)):=\{t(z)=e_{j}(z)+\sum_{l<j}{c_l \,e_{l}(z)}: c_{l}\in \mathbb{C}\}$. Next definition is due to Bloom (\cite{Bl01}) (see also \cite{BBL}).

\begin{defn} \label{cheb1}
Let $K \subset \mathbb{C}^{m}$ be compact and $\theta\in \Delta^{\circ}$ be given. A sequence of polynomials $\{t_{j}\}_{j \in N}$, where $N \subset \mathbb{N}$ is said to be \textit{$\theta$-asymptotically Chebyshev} if \begin{itemize}

  \item $s(j)=|k(j)|\rightarrow \infty$ and $N\in \mathcal{C}_{\theta}$,
  \item $||t_{j}||_{K}^{1/|k(j)|} \rightarrow \tau(K, \theta)$  when $j\rightarrow \infty$.

\end{itemize}
Following \cite{BBL}, a sequence $\{t_{j}\}_{j \in \mathbb{N}}$ is called \textit{asymptotically Chebyshev for $K$} if for any $\theta \in \Delta^{\circ}$, there is a subsequence $N\subset \mathbb{N}$ that satisfies the above three conditions. If the sequence has also the condition that for each $\theta \in \Delta^{\circ}$ and for every sequence of $\beta \in \mathbb{N}^{m}$ with $\lim_{|\beta|\rightarrow \infty}{\frac{\beta}{|\beta|}}=\theta$, one has $\lim{\|t_{j}\|_{K}^{\frac{1}{|\beta(j)|}}}=\tau(K, \theta)$, then we say that $\{t_{j}\}$ is a \textit{Z-asymptotically Chebyshev sequence}.  \end{defn}

As observed in \cite{Bl01} and \cite{BBL}, for every regular compact set, one can find a sequence of $Z$-asymptotically Chebyshev polynomials. The concept of a sequence of $Z$-asymptotically Chebyshev polynomials is a generalization of many other important polynomial types studied in the literature such as Fekete polynomials associated with an array of Fekete points in a compact set $K$, Leja polynomials associated with a sequence of so-called Leja points in a compact set $K$ and $L^{2}(\mu)$-minimal polynomials for a compact set $K$, where $\mu$ is a Bernstein-Markov measure. For some other nice examples, see \cite{BBL}. We will be working with $Z$-asymptotically Chebyshev polynomials, and as also mentioned and investigated in \cite{BBL}, our bases in this paper do not have to be orthonormal either.

Let $\{t_{j}\}$ be a $Z$-asymptotically Chebyshev sequence for $K$. We will make the assumption that there is one $t_{j}$ for every $k \in \mathbb{N}^{m}$ in order to get a basis of polynomials (see also Remark 3.7 of \cite{BBL}). Write \begin{equation}\label{unitch}u_{j}(z):=\frac{t_{j}}{||t_{j}(z)||_{K}}.\end{equation} We always assume that $u_{1}(z)=t_{1}(z)\equiv 1$. In \cite{BBL}, the authors study the following Chebyshev-Bergman functions

\begin{equation}\label{Chbe}
  \Gamma_{n}(z):=\sum_{j=1}^{d_{n}}{|u_{j}(z)|^{2}}.
\end{equation}

 In Proposition 2.3 of \cite{BBL}, by using a Zakharyuta-Siciak type theorem of Bloom (\cite{Bl01}, Theorem 4.2) for $Z$-asymptotically Chebyshev sequences for compact sets in $\mathbb{C}^{m}$ and a diagonalization argument, the authors prove for the sequence (\ref{Chbe}) that when a subsequence $L$ of $\mathbb{N}$ is given, one can find another subsequence $L' \subset L$ and a countable dense subset of points $\{z_{k}\}\,(k=1, 2, \ldots) $\,in\, $\mathbb{C}^{m}$ such that the following holds $$\lim_{n\rightarrow \infty,\,n\in L'}{\frac{1}{2n}\log{\Gamma_{n}(z_{k})}}=V_{K}(z_{k}), \,k=1, 2, \ldots$$

Theorem 4.2 in \cite{Bl01} also yields that the sequence $\{\frac{1}{2n}\log{\Gamma_{n}}\}$ is locally uniformly bounded from above on $\mathbb{C}^{m}$.

As a consequence of Proposition 2.3 in \cite{BBL}, the following lemma is proved, which will be crucial for the purposes of this paper.

\begin{lem}[Corollary 2.6, \cite{BBL}]\label{loc1}
Given a compact set $K \subset \mathbb{C}^{m}$, for a sequence of $Z$-asymptotically Chebyshev polynomials for $K$, one has $$\frac{1}{2n}\log{\Gamma_{n}}\rightarrow V_{K}$$in \,$L^{1}_{loc}(\mathbb{C}^{m})$.
\end{lem}
$\mathcal{D}^{p, q}(\mathbb{C}^{m})$ denotes the space of test forms of bidegree $(p, q)$ on $\mathbb{C}^{m}$.

Our probabilistic setup necessary in the sequel will be as follows. Following the paper \cite{BCM}, we describe how we randomize the space $\mathcal{P}_{n}$.  Let $K\subset \mathbb{C}^{m}$ be compact. Let $\{u_{j}\}_{j=1}^{d_{n}}$ be a basis for $\mathcal{P}_{n}$ such that the polynomials $u_{j}$ are as in (\ref{unitch}). Then, for any polynomial $F \in \mathcal{P}_{n}$ of degree $n$, we have \begin{equation}\label{repch}F(z)=\sum_{l=1}^{d_{n}}{a^{(n)}_{l}\,u_{nl}(z)}:=\langle a^{(n)}, u^{(n)}(z) \rangle \in \mathcal{P}_{n},\end{equation} where $a^{(n)}=(a^{(n)}_{1}, \ldots, a^{(n)}_{d_{n}})\in \mathbb{C}^{d_{n}}$ and $u^{(n)}(z)=(u_{n1}(z), \ldots, u_{nd_{n}}(z)) \in \mathcal{P}_{n}^{d_{n}}$. We identify the space $\mathcal{P}_{n}$ with $\mathbb{C}^{d_{n}}$ and furnish it with a probability measure $\mu_{n}$ satisfying the moment condition below:

We now give a general moment condition: There exist a constant $\alpha\geq 2$ and for every $n\geq 1$ constants $C_{n}=o(n^{\alpha})>0$ such that
\begin{equation} \label{moment} \int_{\mathbb{C}^{d_{n}}}{\bigl\lvert\log{|\langle a, v \rangle|\bigr\rvert^{\alpha}}d\mu_{n}(a)}\leq C_{n}\end{equation}for every $v\in \mathbb{C}^{d_{n}}$ \, with\, $\| v \|=1$. Hence $(\mathcal{P}_{n}, \mu_{n})$ is the probability space consisting of the random polynomials. We also consider the infinite product probability measure $\mu_{\infty}$ induced by $\mu_{n}$, $\mu_{\infty}= \prod_{n=1}^{\infty}{\mu_{n}}$ on the product space $\prod_{n=1}^{\infty}{\mathcal{P}_{n}}$: \begin{equation*}
                                                      (\mathcal{P}_{\infty}, \mathbf{\mu}_{\infty})=(\prod_{n=1}^{\infty}{\mathcal{P}_{n}}, \prod_{n=1}^{\infty}{\mu_{n}}).
                                                    \end{equation*} These probability spaces varying with the degree $n$ depend on the choice of basis, however the equidistribution of zeros of polynomials will be independent of the basis chosen, as Theorem \ref{equidc1} corroborates.

For the practical use of our results, we note that for any regular compact set $K\subset \mathbb{C}^{m}$ there exists a sequence of $Z$-asymptotically Chebyshev polynomials.

Many other widely used probability measures verify this moment condition (\ref{moment}), such as Gaussian, Fubini-Study, locally moderate probability measures etc. Also, we wish to underline that this moment condition provides us with a fairly general probabilistic setting involving no  i.i.d. (independent and identically distributed) coefficients for codimension $1$ and there will be no specific probability distribution function either. In the case of Gaussian and Fubini-Study probability measures, the constant $C_{n}$ above in the moment condition (\ref{moment}) becomes a universal constant. We refer the interested reader to \cite{BCM} for a more detailed exposition of these cases.

The zero set of $F$ is denoted by $Z_{F}$, that is, $Z_{F}:=\{z\in \mathbb{C}^{m}: F(z)=0\}$. For this zero set, we then consider the random current of integration over $Z_{F}$, in symbols $[Z_{F}]$, defined as follows: Fix a test form \,$\varphi\in \mathcal{D}^{m-1, m-1}(\mathbb{C}^{m})$  $$ \langle [Z_{f}], \varphi \rangle := \int_{\mathrm{Reg}Z_{F}}{\varphi},$$where $\mathrm{Reg}Z_{F}$ is the set of regular points of $Z_{F}$. A classical result of Lelong (see, e.g., \cite{Dem12}, Chapter 3, Theorem 2.7), gives that  $[Z_{F}]$ a closed positive $(1, 1)$-current on $\mathbb{C}^{m}$.\\
The expectation and the  variance of the random current $[Z_{F}]$ are given by \begin{align}\label{expvari}
  \mathbb{E} \langle [Z_{F}], \varphi \rangle:=\int_{\mathcal{P}_{n}}{\langle [Z_{F}], \varphi\rangle\, d\mu_{n}(F)} \nonumber \\  \mathrm{Var} \langle [Z_{F}], \, \varphi \rangle := \mathbb{E}\langle [Z_{F}], \, \varphi \rangle^{2}- (\mathbb{E}\langle [Z_{F}], \, \varphi \rangle)^{2},\end{align} where $\varphi\in \mathcal{D}^{m-1, m-1}(\mathbb{C}^{m})$ and $\mu_{n}$ is the probability measure on $\mathbb{C}^{m}$ coming from the identification of $\mathcal{P}_{n}$. Expectation can be regarded as a current-valued random variable as well.

Note that the moment condition (\ref{moment}) is slightly different than the one given in \cite{BCM} and \cite{CM1} (see, e.g., p.3, assumption ($B$) in \cite{BCM}) in order to guarantee that variance of a random current of integration (see section \ref{S3}) is well-defined.
As is well-known, we have the Poincar\'{e}- Lelong formula \begin{equation}\label{pole}
                                                            [Z_{F}]= dd^c \log{|F|},
                                                          \end{equation}
The normalized form $dd^{c}=\frac{i}{\pi} \partial \overline{\partial}$ is used throughout the paper. In the sequel, we study the random currents of integration by normalizing them with the degree of the polynomial, namely, given $F_{n}\in \mathcal{P}_{n}$ of degree $n$ having a representation as in (\ref{repch}), $[\widehat{Z_{F_{n}}}]:=\frac{1}{n}[Z_{F_{n}}]$.

We remark that for simplicity we only work with real-valued test forms throughout the paper but our results are of course true for real-valued continuous forms with compact support in $\mathbb{C}^{m}$ by density of test forms in continuous forms with compact support.

\section{Equidistribution Result}\label{S3}
\subsection{Expected Distribution of zeros}
As before, let a compact set $K\subset \mathbb{C}^{m}$ be given. For a random polynomial $F_{n}\in \mathcal{P}_{n}$ for $K$ as in (\ref{repch}), we have
\begin{lem}\label{expw}
The following holds true
\begin{equation}\label{expd}
 \mathbb{E}[\widehat{Z_{F_{n}}}] \rightarrow dd^c V_{K}
\end{equation} in the weak* topology of currents as the degree $n \rightarrow \infty$.
\end{lem}

\begin{proof}
First, by using the relation (\ref{Chbe}), form the following unit vectors in $\mathbb{C}^{d_{n}}$ \begin{equation}\label{makeu}
\lambda^{(n)}(z):=(\frac{u_{1}(z)}{\sqrt{\Gamma_{n}(z)}}, \ldots, \frac{u_{d_{n}}(z)}{\sqrt{\Gamma_{n}(z)}} ).
\end{equation}Observe that \begin{equation}\label{decomp}
\frac{1}{n}\log{|F_{n}(z)|}=\frac{1}{n}\log{|\langle a, \lambda^{(n)}(z) \rangle|} + \frac{1}{2n}\log{\Gamma_{n}(z)},
\end{equation} here $a^{(n)}=(a^{(n)}_{1}, \ldots, a^{(n)}_{d_{n}})\in \mathbb{C}^{d_{n}}$.
Let us take a test form $\varphi \in \mathcal{D}^{m-1, m-1}(\mathbb{C}^{m})$. We have, by definition of expectation, the Poincare-Lelong formula (\ref{pole}), the identification of $\mathcal{P}_{n}$ with $\mathbb{C}^{d_{n}}$ and Fubini-Tonelli's theorem, \begin{equation}\label{expd2}
  \frac{1}{n}\mathbb{E}\langle [Z_{F_{n}}], \varphi \rangle= \int_{\mathbb{C}^{d_{n}}}{\langle\frac{1}{2n}dd^{c} \log{\Gamma_{n}}, \varphi \rangle d\mu_{n}(a^{(n)})} + \frac{1}{n} \int_{\mathbb{C}^{m}}{\int_{\mathbb{C}^{d_{n}}}{\log{|\langle a^{(n)}, \lambda^{(n)}(z)\rangle|}}d\mu_{n}(a^{(n)}) dd^{c}\varphi(z)}. \end{equation}By our moment condition (\ref{moment}) and H\"{o}lder's inequality, the second term in (\ref{expd2}) can be estimated from above as follows
  \begin{equation}\label{secondterm}
 \frac{1}{n} \int_{\mathbb{C}^{m}}{\int_{\mathbb{C}^{d_{n}}}{\log{|\langle a^{(n)}, \lambda^{(n)}(z)\rangle|}}d\mu_{n}(a^{(n)}) dd^{c}\varphi(z)} \leq \frac{C_{n}^{1/\alpha}}{n} D_{\varphi}, \end{equation} where $D_{\varphi}$ is some finite constant depending on the form $\varphi$ having a compact support in $\mathbb{C}^{m}$, to be specific here, it can be taken to be the sum of the supremum norms of the coefficients of the form $dd^{c}\varphi$. When we pass to the limit as $n\rightarrow \infty$ in (\ref{expd2}), the second term goes to zero owing to the inequality (\ref{secondterm}). Therefore the first term converges to $dd^{c}V_{K}$ in the weak* topology by Lemma \ref{loc1}, which concludes the proof.\end{proof}

\begin{rem}\label{exbo}
   $|\mathbb{E}\langle [Z_{F_{n}}], \varphi \rangle|$ is bounded for any $\varphi \in \mathcal{D}^{m-1, m-1}(\mathbb{C}^{m})$ because $\{\frac{1}{2n}\log{\Gamma_{n}}\}$ is locally uniformly bounded from above on $\mathbb{C}^{m}$ (\cite{BBL}, page 6) and $\varphi$ has a compact support in $\mathbb{C}^{m}$ so, in the expression (\ref{expd2}), the first term is bounded from above and the second integral has, as seen from the proof, already a bound from above (and also from below), which all in all gives the boundedness of $|\mathbb{E} \langle [Z_{F_{n}}], \varphi \rangle|$.

   Observe also that the exponent $\alpha$ in Lemma \ref{expw} does not have to be bigger than or equal to $2$, here the condition $\alpha \geq 1$ works as well, however in the next section, we shall need $\alpha$ to satisfy $\alpha \geq 2$.
\end{rem}
\subsection{Variance Estimate}
We establish a variance estimate of a random zero current of integration over its zero set. We follow the estimation technique used in the proof of Theorem 3.1 in \cite{BG}.

\begin{thm}\label{main}
Assume that the probability space $(\mathcal{P}_{n}, \mu_{n})$ verifies the moment condition (\ref{moment}). Then for any form $\varphi \in \mathcal{D}^{m-1, m-1}(\mathbb{C}^{m})$, the following variance estimate of the random current of integration $[\widehat{Z_{F_{n}}}]$ holds \begin{equation}\label{varas}
    \mathrm{Var}{\langle [\widehat{Z_{F_{n}}}], \varphi \rangle}\leq D_{\varphi}^{2}\,(C_{n})^{\frac{2}{\alpha}} \frac{1}{n^{2}},
\end{equation}where $D_{\varphi}$ is a constant depending on the test form $\varphi$.
\end{thm}

\begin{proof}
   Pick any $\varphi \in \mathcal{D}^{m-1, m-1}(\mathbb{C}^{m})$. By the representation (\ref{repch}) for $F_{n}$ and the Poincare-Lelong formula (\ref{pole}), we first write for the first term of the variance in (\ref{expvari})
\begin{equation}\label{cor}
   \mathbb{E}\langle [\widehat{Z_{F_{n}}}], \, \varphi \rangle^{2}= \frac{1}{n^{2}} \int_{\mathcal{P}_{n}}\int_{\mathbb{C}^{m}}\int_{\mathbb{C}^{m}}{\log{|\langle a^{(n)}, u^{(n)}(z) \rangle|}\log{|\langle a^{(n)}, u^{(n)}(w) \rangle|}dd^{c}\varphi(z)dd^{c}\varphi(w)d\mu_{n}(F_{n})},\end{equation} where $u^{(n)}(z):=(u_{1}(z), \ldots, u_{d_{n}}(z))$. By the relation (\ref{decomp}), the integrand  in (\ref{cor}) becomes
   \begin{align}
\frac{1}{4n^{2}}\log{\Gamma_{n}(z)}\log{\Gamma_{n}(w)}+\frac{1}{2n^{2}} \log{\Gamma_{n}(z)}\log{| \langle a^{(n)}, \lambda^{(n)}(w)\rangle|} + \frac{1}{2n^{2}} \log{\Gamma_{n}(w)}\log{| \langle a^{(n)}, \lambda^{(n)}(z)\rangle|}
\label{cordec}  \\  +\frac{1}{n^{2}}\log{| \langle a^{(n)}, \lambda^{(n)}(z)\rangle|}\log{| \langle a^{(n)}, \lambda^{(n)}(w)\rangle|}. \nonumber \end{align}Considering now the four integrands given in (\ref{cordec}), we write $\mathbb{E}\langle \widehat{[Z_{F_{n}}}], \varphi \rangle^{2}= B_{1} + 2 B_{2} + B_{3}$, where
 \begin{equation}\label{onevar} B_{1}= \int_{\mathcal{P}_{n}}\int_{\mathbb{C}^{m}}\int_{\mathbb{C}^{m}}{\frac{1}{2n}\log{\Gamma_{n}(z)} \frac{1}{2n}\log{\Gamma_{n}(w)}dd^{c}\varphi(z)dd^{c}\varphi(w) d\mu_{n}(F_{n})},\end{equation}

 \begin{equation}\label{twovar}
   B_{2}= \int_{\mathcal{P}_{n}}\int_{\mathbb{C}^{m}}\int_{\mathbb{C}^{m}}{\frac{1}{2n}\log{\Gamma_{n}(z)} \frac{1}{n}\log{|\langle a^{(n)}, \lambda^{(n)}(w)\rangle|}dd^{c}\varphi(z)dd^{c}\varphi(w) d\mu_{n}(F_{n})},
 \end{equation}
 (The second  term and the third one in (\ref{cordec}) are actually the integrands that yield the same result) and \begin{equation}\label{threevar}
                     B_{3}=\int_{\mathcal{P}_{n}}\int_{\mathbb{C}^{m}}\int_{\mathbb{C}^{m}}{\frac{1}{n}\log{|\langle a^{(n)}, \lambda^{(n)}(z)\rangle|} \frac{1}{n} \log{|\langle a^{(n)}, \lambda^{(n)}(w)\rangle|}}dd^{c}\varphi(z) dd^{c}(w) d\mu_{n}(F_{n}). \end{equation}
 From  the locally uniform boundedness of $\{\frac{1}{2n} \log{\Gamma_{n}}\}$ (see the arguments preceding Lemma \ref{loc1}), the moment assumption (\ref{moment}) and the Fubini-Tonelli's theorem, we see that $B_{1}, B_{2}$ and $B_{3}$ are all integrable,

For the second term of variance, by expanding the expectation expression (\ref{expd2}), we have  \begin{equation*}
                                                                                               (\mathbb{E}\langle [\widehat{Z_{F_{n}}}], \varphi \rangle)^{2}=J_{1}+ 2 J_{2}+J_{3},
                                                                                              \end{equation*}where \begin{equation} \label{J1}J_{1}= \big(\frac{1}{4n^{2}}\int_{\mathcal{P}_{n}}\int_{\mathbb{C}^{m}}{\log{|\Gamma_{p}(z)|}dd^{c}\varphi(z)d\mu_{n}(F_{n})} \big)^{2}\end{equation} \begin{equation}\label{J2}
                                                                                                                            J_{2}= \big(\frac{1}{2n}\int_{\mathcal{P}_{n}}\int_{\mathbb{C}^{m}}{\log{|\Gamma_{p}(z)|}dd^{c}\varphi(z)d\mu_{n}(F_{n})} \big) \big( \frac{1}{n} \int_{\mathcal{P}_{n}}\int_{\mathbb{C}^{m}}{\log{|\langle a^{(n)}, \lambda^{(n)}(z)\rangle|}}dd^{c}\varphi(z)d\mu_{n}(F_{n})\big)
                                                                                                                           \end{equation} and \begin{equation}\label{J3}
                                                                                                                             J_{3}= \big( \frac{1}{n^{2}} \int_{\mathcal{P}_{n}}\int_{\mathbb{C}^{m}}{\log{|\langle a^{(n)}, \lambda^{(n)}(z)\rangle|}}dd^{c}\varphi(z)d\mu_{n}(F_{n})\big)^{2}
                                                                                                                           \end{equation}
Note that all of the integrals $J_{1}, J_{2}$ and $J_{3}$ are finite since $\mathbb{E}\langle [\widehat{Z_{F_{n}}}], \varphi \rangle$ is bounded by Remark \ref{exbo}.

 According to the above relations, we get $B_{1}=J_{1}$ and $B_{2}=J_{2}$. Hence, the only integrals that are not killed by each other are $J_{3}$ and $B_{3}$, which are not always equal to each other, so we have \begin{equation}\label{varip}
    \mathrm{Var}\langle [\widehat{Z_{F_{n}}}], \varphi \rangle = B_{3}-J_{3},
    \end{equation}so it will suffice to estimate the term $B_{3}$ from above to obtain the variance estimation. To do this, we apply Hölder's inequality twice using the fitting exponents. Then, by the Hölder's inequality with $\frac{1}{\alpha}+ \frac{1}{\theta}=1$ and proceeding exactly in the same way as above, where $\alpha \geq 2$ is the exponent in the moment condition (\ref{moment}), one first has

\begin{equation}\label{varterm3}B_{3} \leq\int_{\mathbb{C}^{m}}\int_{\mathbb{C}^{m}}{dd^{c}\varphi(z) dd^{c}\varphi(w)} \frac{1}{n^{2}} \int_{\mathbb{C}^{d_{n}}}{|\log{|\langle a^{(n)}, \lambda^{(n)}(z) \rangle|}||\log{|\langle a^{(n)}, \lambda^{(n)}(w) \rangle|}|dd^{c}\varphi(z) dd^{c}\varphi(w) d\mu_{n}(a^{(n)})}.\end{equation} The right-hand side of this last inequality is less than or equal to the following by the Hölder's inequality $$\int_{\mathbb{C}^{m}}\int_{\mathbb{C}^{m}}{dd^{c}\varphi(z) dd^{c}\varphi(w)}\frac{1}{n^{2}} \Big\{ \int_{\mathbb{C}^{d_{n}}}{|\log{|\langle a^{(n)}, \lambda^{(n)}(z) \rangle|}|^{\alpha}d\mu_{n}(a^{(n)})}\Big\}^{\frac{1}{\alpha}} \Big\{ \int_{\mathbb{C}^{d_{n}}}{|\log{|\langle a^{(n)}, \lambda^{(n)}(z) \rangle|}|^{\theta}d\mu_{n}(a^{(n)})}\Big\}^{\frac{1}{\theta}},$$ which gives that

$$B_{3} \leq \int_{\mathbb{C}^{m}}\int_{\mathbb{C}^{m}}{dd^{c}\varphi(z) dd^{c}\varphi(w)} \frac{1}{n^{2}} C_{n}^{\frac{1}{\alpha}} \Big\{ \int_{\mathbb{C}^{d_{n}}}{|\log{|\langle a^{(n)}, \lambda^{(n)}(z) \rangle|}|^{\theta}d\mu_{n}(a^{(n)})}\Big\}^{\frac{1}{\theta}}.$$

    We have to apply Hölder's inequality to the innermost integral once more as we mentioned. Here, the stipulation that $\alpha \geq 2$ (therefore, $\alpha \geq 2 \geq \theta$) is pivotal, since it permits us to reuse the Hölder's inequality, resulting in, \begin{equation}\label{höld2}
                       B_{3} \leq \int_{\mathbb{C}^{m}}\int_{\mathbb{C}^{m}}{dd^{c}\varphi(z) dd^{c}\varphi(w)} \frac{1}{n^{2}} C_{n}^{\frac{2}{\alpha}} \leq \frac{1}{n^{2}}D_{\varphi}^{2} C_{n}^{\frac{2}{\alpha}},
                     \end{equation} which concludes \begin{equation}\label{vars}
                      \mathrm{Var}{\langle [\widehat{Z_{F_{n}}}], \varphi \rangle}\leq D_{\varphi}^{2}\,\frac{C_{n}^{2/\alpha}}{n^{2}},
                    \end{equation}thereby finalizing the variance estimate of the random current of integration $[\widehat{Z_{F_{n}}}]$.\end{proof}

Now the equidistribution result in codimension $1$ will be proved.

\begin{thm}\label{equidc1}
 Under the same condition as in Theorem \ref{main}, if $\sum_{n=1}^{\infty}{\frac{C_{n}^{2/\alpha}}{n^{2}}}< \infty$, then for $\mu_{\infty}$-almost every sequence $\textbf{F}=\{F_{n}\}$, \begin{equation}\label{equidc}
                                                      [\widehat{Z_{F_{n}}}] \rightarrow dd^{c} V_{K}
                                                    \end{equation}in the weak* topology of currents as $n\rightarrow \infty$.
\end{thm}

\begin{proof}
 Let $\varphi \in \mathcal{D}^{m-1, m-1}(\mathbb{C}^{m})$. Take a random sequence $\textbf{F}=\{F_{n}\}^{\infty}_{n=1}$. Following \cite{SZ99},  we consider the random variables \begin{equation}\label{newr}
  W_{n}(\textbf{F}):= ([\widehat{Z_{F_{n}}}]- \mathbb{E}[\widehat{Z_{F_{n}}}], \varphi)^{2} \geq 0.\end{equation}First, by the alternative definition of variance, \begin{equation}\label{sumson}
     \int_{\mathcal{P}_{\infty}}{W_{n}(\textbf{F})d\mu_{\infty}(\textbf{F})}=\mathrm{Var}([\widehat{Z_{F_{n}}}], \varphi).
  \end{equation}By assumption, Theorem \ref{main} and the Beppo-Levi theorem, one has \begin{equation}\label{vars2}
                                        \int_{\mathcal{P}_{\infty}}{\sum_{n=1}^{\infty}{W_{n}(\textbf{F})}d\mu_{\infty}(\textbf{F})}=\sum_{n=1}^{\infty}{\int_{\mathcal{P}_{\infty}}{W_{n}(\textbf{F})d\mu_{\infty}(\textbf{F})}}=\sum_{n=1}^{\infty}{\mathrm{Var}([\widehat{Z_{F_{n}}}], \varphi)}< +\infty,
                                      \end{equation}This yields, for $\mu_{\infty}$-almost all $\textbf{F}$, that $W_{n}(\textbf{F}) \rightarrow 0$, namely $$\langle [\widehat{Z_{F_{n}}}], \varphi \rangle- \mathbb{E}\langle [\widehat{Z_{F_{n}}}], \varphi \rangle \rightarrow 0$$ $\mu_{\infty}$-almost surely. This then gives, by Lemma \ref{expw}, for $\mu_{\infty}$-almost all $\textbf{F}=\{F_{n}\}\in \mathcal{P}_{\infty}$, $$[\widehat{Z_{F_{n}}}]\rightarrow dd^{c}V_{K}$$ in the weak* topology of currents, which finishes the proof.
\end{proof}

All the probability measures such as the standard Gaussian, the Fubini-Study, area measure of spheres etc. considered in Section 4 of \cite{BCM} satisfy the moment condition (\ref{moment}), and so Theorem \ref{equidc1} is a generalization for all of these cases. As we have mentioned in the introduction, only the standard Gaussian and the Fubini-Study probability measures verify this logarithmic moment integral with universal constants.

\section{Equidistribution in higher codimensions}\label{sec3}

In Section \ref{S3}, we have obtained an equidistribution result in codimension $1$ for random polynomials having a representation with respect to the bases comprising of $Z$-asymptotically Chebyshev polynomials. Based on these considerations, in order to prove such an assertion for codimensions bigger than $1$, as is well-known from the pluripotential theory and, in particular, from the continuity properties of complex Monge-Ampere operator, we need more than $L^1_{loc}$-convergence of Chebyshev-Bergman type functions $\frac{1}{2n}\log{\Gamma_{n}}$ to $V_{K}$, as given in Lemma \ref{loc1}. At the moment, we do not have such a stronger convergence result.

 In this final section, we focus on a considerably better researched classical context for the equidistribution problem, namely the orthogonal basis setup. From here on, we shall consider the orthogonal bases. Let $K\subset \mathbb{C}^{m}$ be a compact set. $K$ is said to be \textit{locally regular} if for every $z \in K$, the pluricomplex Green function $V_{K \cap \overline{B(z, r)}}$ is continuous at $z$ for a sufficiently small $r=r(z)>0$. Let $q$ be a weight function, that is to say, a continuous real-valued function on $K$. In line with the unweighted case, the weighted extremal function for the pair $(K, q)$ is defined as below $$V_{K, q}(z):=\sup{\{v(z): v\in \mathcal{L}(\mathbb{C}^{m}), v\leq q\,\,on\,\,K\}}.$$

There are important properties of the function $V_{K, q}$, for example, for a locally regular compact set $K$ and a weight function $q$ on $K$, we have $V_{K, q}$ is continuous on $\mathbb{C}^{m}$. The Siciak-Zakharyuta representation of the weighted extremal function is given by $$\varphi_{n}(z):=\sup{\{|p(z)|: p\in \mathcal{P}_{n}, \|p e^{-nq}\|_{K} \leq 1\}}.$$  When $K$ is locally regular and $q$ is a weight function, then $$\lim_{n\rightarrow \infty}{\frac{1}{n}\log{\varphi_{n}(z)}}=V_{K, q}(z).$$For further information, we refer the reader, for example, to the paper \cite{BL07}.

Let $\{\sigma_{n}\}$ be a sequence of probability measures on $K \subset \mathbb{C}^{m}$ such that, for $n=1, 2, \ldots,$ the following holds $\|P e^{-nq}\|_{K} \leq M_{n} \|P e^{-nq}\|_{L^{2}(\sigma_{n})}$ for all $P\in \mathcal{P}_{n}$ with $\lim_{n\rightarrow \infty}{M_{n}^{1/n}}=1$. These measures are called \textit{asymptotically weighted Bernstein-Markov measures} for the pair $(K, q)$. There are nice examples of this kind of measures in Section 2.2 of \cite{BBL}. Now given a locally regular compact set $K$ in $\mathbb{C}^{m}$ and a sequence $\{\sigma_{n}\}$ of asymptotically weighted Bernstein-Markov probability measures, writing $B_{n}(z):=\sum_{j=1}^{d_{n}}{|p_{nj}(z)|^{2}}$, where $\{p_{n1}, \ldots, p_{nd_{n}}\}$ is an orthogonal basis in $L^{2}(e^{-2nq}\sigma_{n})$ for $\mathcal{P}_{n}$ with $\|p_{nj}e^{-nq}\|_{K}=1$,  it is proven in Proposition 2.8 of \cite{BBL} that \begin{equation}\label{szsz}\lim_{n\rightarrow \infty}{\frac{1}{2n}\log{B_{n}(z)}}=V_{K, q}\end{equation} locally uniformly on $\mathbb{C}^{m}$. In the unweighted setting, i.e., when $q=0$, see Proposition 2.9 in the same paper.

 Unlike the previous section, $\mu_{n}$ on the polynomial space $\mathcal{P}_{n}$, this time, is induced by the probability distribution law $\mathbf{P}$ of the i.i.d. random coefficients $a_{l}$ in the representation (\ref{repch}) with a density $\varphi: \mathbb{C} \rightarrow [0, N]$ satisfying the property that there are constants $\epsilon>0$ and $\theta >2m$ such that \begin{equation}\label{logtail}\mathbf{P}(\{z\in \mathbb{C}: \log|z|>R\})\leq \frac{\epsilon}{R^{\theta}},\,\,\,\forall R\geq 1.\end{equation} This kind of density was studied in \cite{Bay16} and \cite{BCM}. This choice of probability distribution includes real or complex Gaussian distributions. The authors in \cite{BCM} (Lemma 4.15 there) show that the  measures $\mu_{n}$ verify the moment condition (\ref{moment}) with the upper bound \begin{equation}\label{dnmom} B d_{n}^{\alpha/ \theta}\end{equation} ($B=B(N, \alpha, \theta, \epsilon)$)\,for any constant $\alpha$ with $1\leq \alpha < \theta$, which gives us that, under this probabilistic setting, with the ideas we use in the previous section, the analogoues of Lemma \ref{expw}, Theorem \ref{main} and Theorem \ref{equidc1} can be seen to be true for asymptotically weighted Bernstein-Markov probability measures. A similar probability distribution function was also considered by Bloom and Levenberg, see \cite{BL15} for further details.

 We work with the random polynomial mappings $G^{k}_{n}=(F^{1}_{n}, \ldots, F^{k}_{n}): \mathbb{C}^{m} \rightarrow \mathbb{C}^{k}$ of degree $n$ and here \begin{equation}\label{indv}
                                                                    F^{j}_{n}=\sum_{l=1}^{d_{n}}{a^{j}_{nl}\,p^{j}_{nl}}= \langle a^{j}_{n}, p^{j}_{n}(z) \rangle\in \mathcal{P}_{n},\end{equation}where $a^{j}_{n}=(a^{j}_{n1}, \ldots, a^{j}_{nd_{n}}) \in \mathbb{C}^{d_{n}}$ and $p^{j}_{n}(z)=(p^{j}_{n1}(z), \ldots, p^{j}_{nd_{n}}(z)) \in \mathcal{P}_{n}^{d_{n}}$. Also, as was written in (\ref{makeu}), \begin{equation}\label{makeu2}
                                                                      \beta^{j}_{n}(z):=(\frac{p^{j}_{n1}(z)}{\sqrt{B_{n}(z)}}, \ldots, \frac{p^{j}_{nd_{n}}(z)}{\sqrt{B_{n}(z)}})
                                                                    \end{equation} are the unit vectors in $\mathbb{C}^{d_{n}}$. Accordingly, we take the product probability spaces $(\mathcal{P}^{k}_{n}, \mu^{k}_{n})$ into consideration, where $\mu^{k}_{n}:=\mu_{n}\times \ldots \times \mu_{n}$ is the $k^{th}$ product measure. We will use the infinite product probability spaces for sequences $\{G^{k}_{n}\}_{n=1}^{\infty}$ of polynomial mappings: $(\mathcal{P}^{k}_{\infty}, \mu_{\infty}^{k})=(\prod_{n=1}^{\infty}{\mathcal{P}^{k}_{n}}, \,\prod_{n=1}^{\infty}{\mu^{k}_{n}})$.

By using Bertini's theorem, for any generic choice of $F^{1}_{n}, \ldots, F^{k}_{n}$ of polynomials, their zero divisors $Z_{F^{1}_{n}}, \ldots, Z_{F^{k}_{n}}$ are smooth and intersect transversally. Since the probability measures $\mu_{n}$ are absolutely continuous with respect to the Lebesgue measure, as an application of Sard's theorem, by using the argument in the proof of Proposition 4.1 of \cite{CM1}, for almost every choice of $(F^{1}_{n}, \ldots, F^{k}_{n}) \in \mathcal{P}^{k}_{n}$, the zero divisors $Z_{F^{k}_{1}}, \ldots, Z_{F^{k}_{n}}$ are smooth and intersect transversally, and so the wedge product of the currents $Z_{F^{1}_{n}}, \ldots, Z_{F^{k}_{n}}$, which is denoted by $[Z_{G^{k}_{n}}]$, is well-defined

 \begin{equation*}
  [Z_{G^{k}_{n}}]:= [Z_{F^{1}_{n}}] \wedge \ldots \wedge [Z_{F^{k}_{n}}]
  \end{equation*}

 By the method followed in [\cite{Gun}, arguments after Theorem 6.5] (see also [\cite{BG}, Theorem 4.2]) whose main idea is based on \mbox{[\cite{CLMM}, Proposition 3.5]} and utilizing the limit (\ref{szsz}) in the relevant parts, we conclude the following.

\begin{lem}\label{sonnn}
Given that the polynomial mappings $$G^{k}_{n}=(F^{1}_{n}, \ldots, F^{k}_{n}): \mathbb{C}^{m} \rightarrow \mathbb{C}^{k},$$ we have \begin{equation}\label{decex}
                                                              \mathbb{E}[Z_{G^{k}_{n}}]=\mathbb{E}[Z_{F^{1}_{n}}] \wedge \ldots \wedge \mathbb{E}[Z_{F^{k}_{n}}]
                                                            \end{equation} for $k=1, \ldots, m$.  Moreover, \begin{equation}\label{codimm}
                                                                                       \mathbb{E}[Z_{G^{k}_{n}}] \rightarrow (dd^{c} V_{K, q})^{k}
                                                                                     \end{equation}in the weak* topology of currents as $n \rightarrow \infty$.
\end{lem}

In order to achieve an equidistribution result for the random currents of integration $[Z_{G^{k}_{n}}]$, we use the techniques used in the proof of Theorem 3.1 of \cite{Shif} and Theorem 1.1 from \cite{BG} which gives the variance estimation of $[Z_{G^{k}_{n}}]$ associated with the zero sets of $k$-tuples of random holomorphic sections. It is proven by induction on the codimension $k$. The key concepts within the proof are the restriction of the current $[Z_{G^{k}_{n}}]=[Z_{F^{1}_{n}}] \wedge \ldots \wedge [Z_{F^{k}_{n}}]$ to the zero set $Z_{G^{k-1}_{n}}$ and the zero set $Z_{F^{k}_{n}}$, combined with applying the variance estimation for codimension $1$.

\begin{thm}\label{var2}

   If the probability space $(\mathcal{P}_{n}^{k}, \mu_{n}^{k})$ has the aforementioned conditions, then for any fixed $\varphi \in \mathcal{D}^{m-k, m-k}(\mathbb{C}^{m})$, we have the estimate $\mathrm{Var}{\langle [\widehat{Z_{G^{k}_{n}}}], \varphi \rangle}\leq D_{\varphi}^{2}\,C_{n}^{\frac{2}{\alpha}} \frac{1}{n^{2}}$, where $D_{\varphi}$ is a constant that is dependent on the test form $\varphi$.

\end{thm}

Since $d_{n}=\binom{n+m}{m}$, we can find a constant $C>0$ such that $d_{n} \leq C\,n^{m}$, so, by using (\ref{dnmom}) and the inequality  $((d_{n})^{\alpha/\theta})^{2/\alpha} \leq (C\,n^{m})^{\alpha/\theta})^{2/\alpha}= C^{2/ \theta} n^{2m/ \theta}$, we have $(B d^{\alpha/\theta}_{n})^{2/\alpha} \leq B^{2 / \alpha} C^{2/ \theta} n^{2m/\theta}$. Now, by writing $C_{n}:=B (C n^{m})^{\alpha/\theta}$, we see that \begin{equation}\label{sumi}\sum_{n=1}^{\infty}{\frac{C_{n}^{2/\alpha}}{n^{2}}} < \infty\end{equation} because $\theta >2m$. When revising the approach used in Theorem \ref{equidc1} to apply to the zero locus $Z_{G^{k}_{n}}$, by employing Lemma \ref{sonnn} and (\ref{sumi}), we deduce an equidistribution result for any codimension $k$ without any summability condition. This is a slight generalization of Theorem 1.2 in \cite{Bay16} regarding the Bernstein-Markov measures involved.

\begin{thm}
Let $K\subset \mathbb{C}^{m}$ be a locally regular compact set and $q$ is a weight function on $K$. According to the above setup with the sequence $\{\sigma_{n}\}$ of asymptotically weighted Bernstein-Markov probability measures and $L^2(e^{-2qn}\sigma_{n})$-orthogonal basis, we have, for $\mu^{k}_{\infty}$-almost every sequence $\{G^{k}_{n}\}$ \begin{equation}\label{equidc}
                                                      [\widehat{Z_{G^{k}_{n}}}] \rightarrow (dd^{c} V_{K, q})^{k}
                                                    \end{equation}in the weak* topology of currents as $n\rightarrow \infty$, where $k=1, 2, \ldots, m$.
\end{thm}

\textbf{Ethical Approval} Ethics approval was not required for this study.

\textbf{Availability of Data and Materials .} Data sharing was not applicable to this article as no datasets were generated or analyzed during the current study.

\textbf{Competing Interests.} There are no competing financial interests to influence the work in this paper.

\textbf{Funding.} The author of this paper was supported by the TÜBİTAK-2518 project, No: 119N642.

\textbf{Acknowledgments.} We thank the anonymous referee for his/her careful review, which have improved the presentation of this paper. The author thanks Afrim Bojnik for valuable discussions.

%%%%%%%%%%%%%%%%%%%%%%%%%%%%%%%%%%%%%%%%%%%%%%%%%%%%%%%%%%%

{}

 \end{document}